\documentclass[11pt,reqno]{amsart}

\usepackage{amsmath,amsfonts,amsthm,amscd,amssymb,graphicx}
\usepackage[utf8]{inputenc}
\usepackage{amsbsy}
\usepackage{graphicx}
\usepackage{hyperref}
\usepackage[margin=1in]{geometry}
\usepackage{subcaption}

\usepackage{tikz}
\usetikzlibrary{decorations.shapes}
\tikzset{paint/.style={ draw=#1!50!black, fill=#1!50 },
    decorate with/.style=
    {decorate,decoration={shape backgrounds,shape=#1,shape size=2mm}}}

\usepackage{todonotes}

\definecolor{skyblue}{rgb}{0.85,0.85,1}

\newcommand\nod{\nu} 

\newtheorem{theorem}{Theorem}

\newtheorem{define}{Definition}
\newtheorem{lemma}{Lemma}

\newtheorem{obs}{Observation}
\newtheorem{hyp}{Hypothesis}

\newcommand{\bbN}{\mathbb{N}}             
\newcommand{\bbR}{\mathbb{R}}             
     
\newcommand{\p}{\partial}				

\newcommand{\pO}{{\partial \Omega}}

\DeclareMathOperator{\Mor}{Mor}

\DeclareMathOperator{\spec}{spec}

\begin{document}
\title{Nodal deficiency, spectral flow, and the Dirichlet-to-Neumann map}
\author{Gregory
  Berkolaiko}\email{berko@math.tamu.edu}\address{Department of
  Mathematics, Texas A\&M University, College Station, TX 77843-3368, USA}
\author{Graham Cox}\email{gcox@mun.ca}\address{Department of Mathematics and Statistics, Memorial University of Newfoundland, St. John's, NL A1C 5S7, Canada}
\author{Jeremy L. Marzuola}\email{marzuola@email.unc.edu}\address{Dept. of Mathematics, University of North Carolina at Chapel Hill, CB 3250 Phillips Hall, Chapel Hill, NC 27599-3250, USA}

\maketitle
\begin{abstract}
It was recently shown that the nodal deficiency of an eigenfunction is encoded in the spectrum of the Dirichlet-to-Neumann operators for the eigenfunction's positive and negative nodal domains. While originally derived using symplectic methods, this result can also be understood through the spectral flow for a family of boundary conditions imposed on the nodal set, or, equivalently, a family of operators with delta function potentials supported on the nodal set. In this paper we explicitly describe this flow for a Schr\"odinger operator with separable potential on a rectangular domain, and determine a mechanism by which lower energy eigenfunctions do or do not contribute to the nodal deficiency.
\end{abstract}


\section{Introduction}
Let $\Omega \subset \bbR^d$ be a bounded domain with sufficiently smooth boundary, and denote by $\lambda_1 < \lambda_2 \leq \lambda_3 \leq \cdots$ the eigenvalues of the Laplacian, with eigenfunctions $\phi_1, \phi_2, \ldots$, where we have imposed either Dirichlet or Neumann boundary conditions on $\pO$. As in Sturm--Liouville theory, one is often interested in quantifying the oscillation of $\phi_k$ in terms of the index $k$.

The \textit{nodal domains} of $\phi_k$ are the connected components of the set $\{\phi_k \neq 0\}$. We denote the total number of nodal domains by $\nod(\phi_k)$. Courant's nodal domain theorem says that $\phi_k$ has at most $k$ nodal domains \cite{CH53}. In other words, the \emph{nodal deficiency}
\begin{align}
	\delta(\phi_k) := k - \nod(\phi_k)
\end{align}
is nonnegative. Beyond this, however, little is known. While it has been shown that the deficiency only vanishes for finitely many $k$ \cite{P56}, it is generally very difficult to compute, or even estimate.

In \cite{BKS12} the first author, Kuchment and Smilansky gave an explicit formula for the nodal deficiency as the Morse index of an energy functional defined on the space of equipartitions of $\Omega$. More recently \cite{CJM2}, the second two authors, with Jones, computed the nodal deficiency in terms of the spectra of Dirichlet-to-Neumann operators using Maslov index tools developed in \cite{CJM14,DJ11}. In particular, for a simple eigenvalue $\lambda_k$ with Lipschitz nodal domains, it was shown that
\begin{align}\label{def:simple}
	\delta(\phi_k) = \Mor\left( \Lambda_+(\epsilon) + \Lambda_-(\epsilon) \right)
\end{align}
for sufficiently small $\epsilon>0$, where $\Lambda_\pm(\epsilon)$ denote the Dirichlet-to-Neumann maps for the perturbed operator $\Delta +  (\lambda_k + \epsilon)$, evaluated on the positive and negative nodal domains $\Omega_\pm = \{\pm\phi_k > 0\}$, and $\Mor$ denotes the \emph{Morse index}, or number of negative eigenvalues.  For more on the spectrum of Dirichlet-to-Neumann operators, see \cite{AM12,friedlander1991some,M91} and the recent survey \cite{girouard2017spectral}.  

Similarly, if $\phi_*$ is an eigenfunction for a degenerate eigenvalue $\lambda_*$, the same argument yields
\begin{align}\label{def:degenerate}
	\delta(\phi_*) = 1 - \dim \ker (\Delta + \lambda_*) + \Mor\left( \Lambda_+(\epsilon) + \Lambda_-(\epsilon) \right).
\end{align}
Note that the Dirichlet-to-Neumann maps depend explicitly on the
choice of eigenfunction $\phi_* \in \ker (\Delta + \lambda_*)$. In
defining the nodal deficiency of $\phi_*$, we let $k=k_* = \min\{n \in
\bbN : \lambda_n = \lambda_*\}$.

Equations \eqref{def:simple} and \eqref{def:degenerate} remain valid
for the Schr\"odinger operator $L=-\Delta+V$ with sufficiently regular
potential, for instance $V \in L^\infty(\Omega)$.  These formulas were originally obtained from a general spectral
decomposition formula, derived using symplectic methods in
\cite{CJM2}. In Section \ref{sec:flow} we give a more direct proof using spectral flow. For a fixed $\phi_*$ we construct a monotone family of selfadjoint operators $\{L_\sigma\}_{\sigma \geq 0}$, starting at $L_0 = L$, such that the nodal deficiency of $\phi_*$ equals the number of eigenvalue curves for $L_\sigma$ that pass through $\lambda_* + \epsilon$ for some $\sigma>0$; see Figure \ref{fig:qgraphrect} for an illustration.
This invites
a question of potentially great significance: {\it what properties
of the eigenpair $(\lambda_j, \phi_j)$, $\lambda_j \leq \lambda_*$, determine whether the
corresponding spectral flow curve will cross $\lambda_{*}+\epsilon$ and thus
contribute to the nodal deficiency of the eigenfunction $\phi_{*}$?}

\begin{figure}
\includegraphics[width=0.45\textwidth]{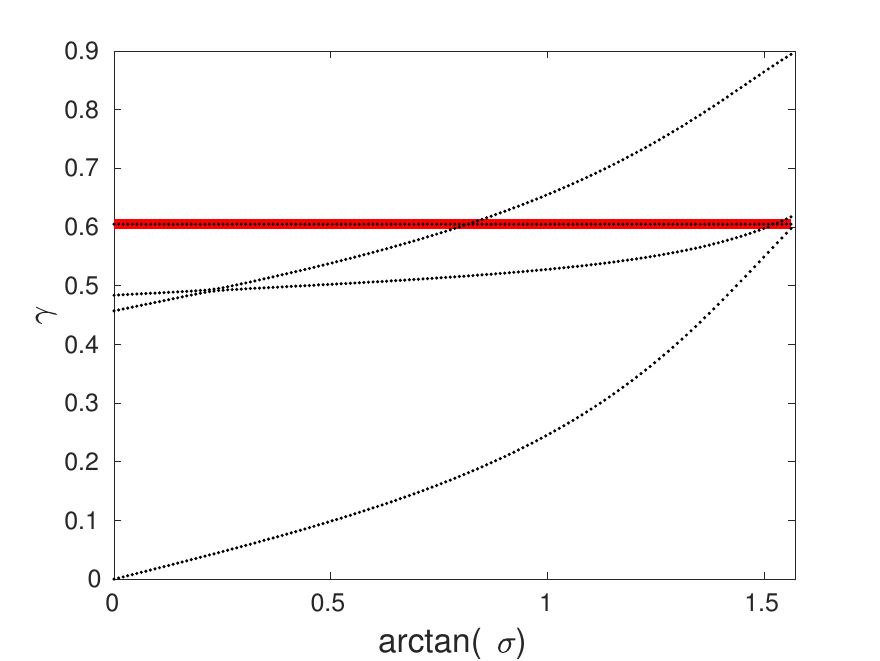}
\includegraphics[width=0.45\textwidth]{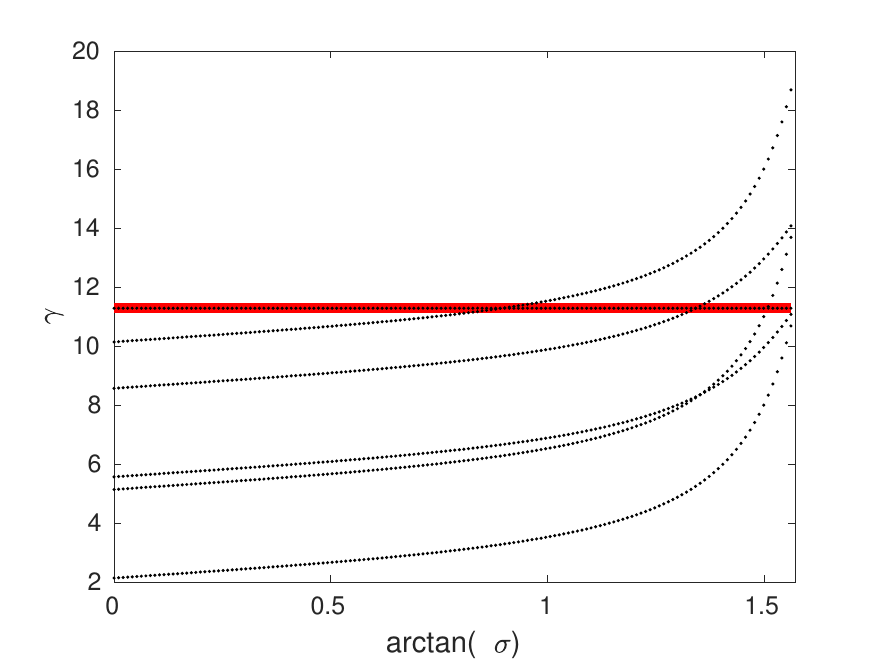}
\caption{Numero-analytic solution of the spectral flow on the tetrahedron quantum graph (left) and on a rectangle (right), as described in Appendix \ref{app:rec}. In both cases the number of curves crossing
$\lambda_*+\epsilon$ matches the nodal deficiency (2 on the left and 3 on the right).}
\label{fig:qgraphrect}
\end{figure}

The main result of this paper is a beautifully geometric answer to this question on rectangular domains, illustrated in Figure~\ref{fig:ellipse}.  Informally speaking, the intersecting
curves arise from the eigenvalues corresponding to the points within
the ellipse but outside the rectangle (both regions are
specified by $(\lambda_{*},\phi_{*})$). This geometric interpretation of the nodal deficiency on a rectangle appeared in \cite{BBF17} (see in particular Figure 10). The advantage of our construction is that it describes precisely how these lattice points contribute to the nodal deficiency, through a mechanism (the spectral flow) which is defined on any domain. Before we explain the precise
meaning of this statement, we mention that for non-separable
problems the situation is likely to be far more complicated due to the
presence of avoided crossings; for example the ``intersection'' around
$\arctan(\sigma)=0.2$ on Figure~\ref{fig:qgraphrect}(left) is in fact
an avoided crossing; see Figure~\ref{fig:graphzoom}.

\begin{figure}[!tbp]
\begin{subfigure}[b]{0.4\textwidth}
	\begin{tikzpicture}[scale=0.9]
		\draw[thick] (0,3) -- (5,3); 
		\draw[thick] (5,0) -- (5,3); 
		\draw[thick] (6.5,0) arc (0:90:6.5 and 4.7);
		\draw[very thick,->] (0,0) -- (7,0); 
		\draw[very thick,->] (0,0) -- (0,5); 
		\draw (1,1) circle[radius=2pt];
		\draw (1,2) circle[radius=2pt];
		\draw (1,3) circle[radius=2pt];
		\fill (1,4) circle[radius=2pt];
		\draw (2,1) circle[radius=2pt];
		\draw (2,2) circle[radius=2pt];
		\draw (2,3) circle[radius=2pt];
		\fill (2,4) circle[radius=2pt];
		\draw (3,1) circle[radius=2pt];
		\draw (3,2) circle[radius=2pt];
		\draw (3,3) circle[radius=2pt];
		\fill (3,4) circle[radius=2pt];
		\draw (4,1) circle[radius=2pt];
		\draw (4,2) circle[radius=2pt];
		\draw (4,3) circle[radius=2pt];
		\draw (4,4) circle[radius=2pt];
		\draw (5,1) circle[radius=2pt];
		\draw (5,2) circle[radius=2pt];
		\draw (5,3) circle[radius=2pt];
		\draw (5,4) circle[radius=2pt];
		\fill (6,1) circle[radius=2pt];
		\draw (6,2) circle[radius=2pt];
		\draw (6,3) circle[radius=2pt];
		\draw (6,4) circle[radius=2pt];
	\end{tikzpicture}
	\caption{simple eigenvalue}
\end{subfigure}
\begin{subfigure}[b]{0.4\textwidth}
	\begin{tikzpicture}[scale=0.9]
		\draw[thick] (0,3) -- (5,3); 
		\draw[thick] (5,0) -- (5,3); 
		\draw[thick] (6.65,0) arc (0:90:6.65 and 4.6);
		\draw[very thick,->] (0,0) -- (7,0); 
		\draw[very thick,->] (0,0) -- (0,5); 
		\draw (1,1) circle[radius=2pt];
		\draw (1,2) circle[radius=2pt];
		\draw (1,3) circle[radius=2pt];
		\fill (1,4) circle[radius=2pt];
		\draw (2,1) circle[radius=2pt];
		\draw (2,2) circle[radius=2pt];
		\draw (2,3) circle[radius=2pt];
		\fill (2,4) circle[radius=2pt];
		\draw (3,1) circle[radius=2pt];
		\draw (3,2) circle[radius=2pt];
		\draw (3,3) circle[radius=2pt];
		\fill (3,4) circle[radius=2pt];
		\draw (4,1) circle[radius=2pt];
		\draw (4,2) circle[radius=2pt];
		\draw (4,3) circle[radius=2pt];
		\draw (4,4) circle[radius=2pt];
		\draw (5,1) circle[radius=2pt];
		\draw (5,2) circle[radius=2pt];
		\draw (5,3) circle[radius=2pt];
		\draw (5,4) circle[radius=2pt];
		\fill (6,1) circle[radius=2pt];
		\draw (6,2) circle[radius=3pt];
		\fill (6,2) circle[radius=2pt];
		\draw (6,3) circle[radius=2pt];
		\draw (6,4) circle[radius=2pt];
	\end{tikzpicture}
	\caption{degenerate eigenvalue}
\end{subfigure}
\caption{Illustrating the result of Observations \ref{obs:rectangle1} and \ref{obs:rectangle2}. For the simple eigenvalue (A), the nodal deficiency is 4, which equals the Morse index of $\Lambda_+ + \Lambda_-$. The degenerate eigenvalue (B) also has nodal deficiency 4. The point $(6,2)$ generates an additional negative eigenvalue of $\Lambda_+ + \Lambda_-$ but does not contribute to the nodal deficiency.}
\label{fig:ellipse}
\end{figure}




Consider the rectangular domain
$R_\alpha = [0,\alpha\pi] \times [0, \pi]$ with $\alpha > 0$. We first
illustrate our result for the Laplacian, where the computations can be
done explicitly. The general statement is formulated and proved in
Section \ref{sec:rectangle}. The spectrum of $-\Delta$ with Dirichlet
boundary conditions on $R_\alpha$ is in one-to-one correspondence with
the points of $\bbN^2$, namely
\begin{equation}
\label{spec:LRa}
	\sigma(-\Delta) = \left\{  \left( \frac{m}{\alpha} \right)^2 + n^2 \ : \ m,n \in \bbN \right\}.
\end{equation}

For a given eigenvalue $\lambda_* = (m_*/\alpha )^2 + n_*^2$, we have $\lambda_* = \lambda_{k_*}$, where
\[
	k_* = \# \left\{(m,n)  : (m/\alpha)^2 + n^2 < \lambda_* \right\} + 1.
\]
This counts the lattice points in the region bounded by the quarter ellipse
\begin{equation}
	E_{\lambda_*} = \left\{ (x,y) \ : \ x > 0, \ y > 0, \
          (x/\alpha)^2 + y^2 < \lambda_* \right\},
\end{equation}
plus the point $(m_*,n_*)$, which lies on the ellipse. On the other hand, the corresponding eigenfunction $\sin (m_* x/\alpha) \sin (n_* y)$ has $m_* n_*$ nodal domains, which coincides with the number of lattice points contained in the rectangle
\begin{equation}
	R_{\lambda_*} = \left\{ (x,y) \ : \  0 < x \leq m_*,\ 0 < y \leq n_* \right\}.
\end{equation}
That is, the nodal deficiency equals the number of lattice points under the ellipse but outside the rectangle, as illustrated in Figure \ref{fig:ellipse}.

\begin{obs}\label{obs:rectangle1}
The nodal deficiency of the $(m_*, n_*)$ eigenfunction is equal to the number of lattice points in the region $E_* \setminus R_*$.
\end{obs}


This holds whether or not $\lambda_*$ is simple. When $\lambda_*$ is simple, we conclude from \eqref{def:simple} that the Morse index of $\Lambda_+(\epsilon) + \Lambda_-(\epsilon)$ equals the number of lattice points in $E_* \setminus R_*$. On the other hand, when $\lambda_*$ is degenerate, $\Lambda_+(\epsilon) + \Lambda_-(\epsilon)$ has an additional $\dim \ker (\Delta + \lambda_{*}) - 1$ negative eigenvalues, according to \eqref{def:degenerate}. This coincides with the number of lattice points on the ellipse, as shown in Figure~\ref{fig:ellipse}(B).

\begin{obs}\label{obs:rectangle2}
The Morse index of $\Lambda_+(\epsilon) + \Lambda_-(\epsilon)$ is equal to the number of lattice points in the region $\overline{E_*} \setminus R_*$.
\end{obs}

Using the spectral flow, we prove that this is not just a numerical coincidence---it is precisely the eigenvalues corresponding to points in $\overline{E_*} \setminus
R_*$ (via equation (\ref{spec:LRa})) that give rise to the spectral flow
curves which cross $\lambda_*+\epsilon$ and thus generate negative
eigenvalues of $\Lambda_+(\epsilon) + \Lambda_-(\epsilon)$.
In Section \ref{sec:rectangle} we formalize this statement as Theorem
\ref{thm:main} and prove it.  The result is valid for any Schr\"odinger
operator with separable potential, and hence does not rely on having
explicit formulas for the eigenvalues and eigenfunctions, as was the
case above.

The spectral flow method can be easily generalized to other settings,
such as Schr\"odinger operators on manifolds and metric
graphs. Figure~\ref{fig:qgraphrect}(left) shows the results
of a numerical computation of eigenvalues of $L_\sigma$ defined on the
nodal set of a deficiency 2 eigenfunction of a metric graph (see, for
instance, \cite{berkolaiko2017elementary} for an accessible
introduction to the subject). Figure~\ref{fig:qgraphrect}(right)
shows a similar computation for a deficiency $3$ eigenfunction of a
rectangular domain.  In both cases, the number of curves crossing
$\lambda_*+\epsilon$ matches the nodal deficiency. As above, the 
main issue it to determine which eigenpairs $(\lambda_j, \phi_j)$ are responsible
for the nodal deficiency.

\begin{figure}
  \includegraphics[width = 0.45\textwidth]{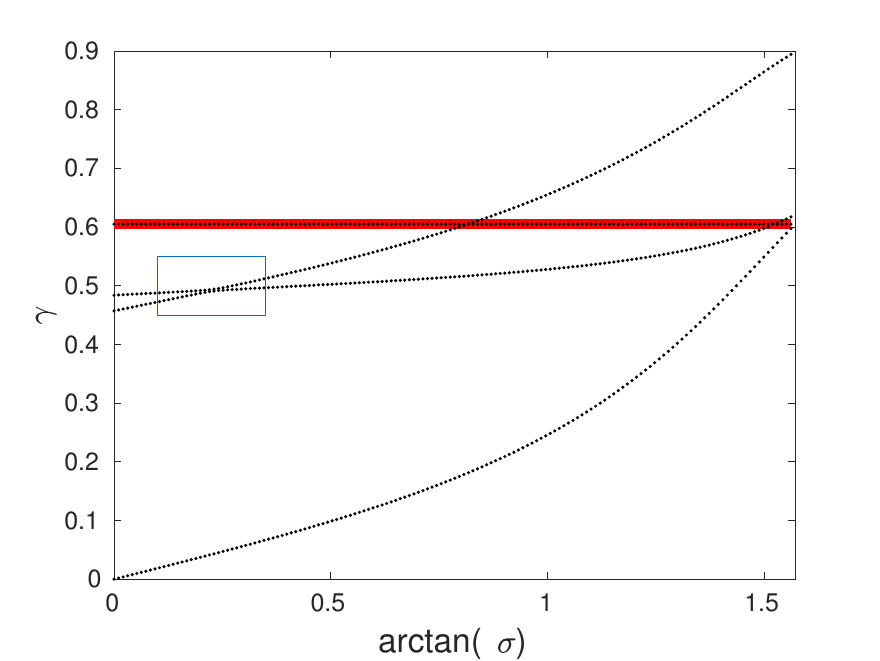}
  \includegraphics[width = 0.45\textwidth]{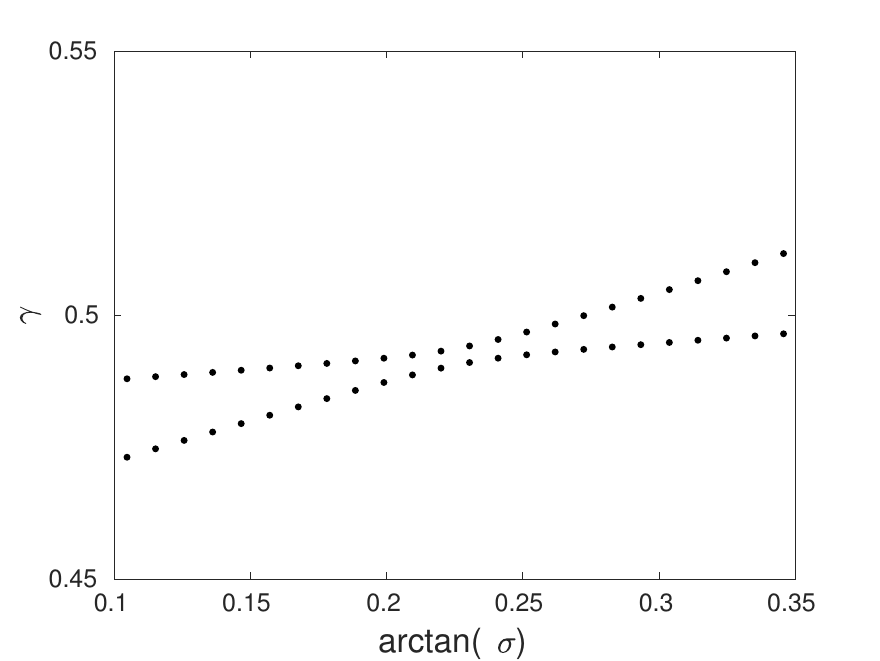}
  \caption{Suspected intersection of flow curves for a quantum
    graph and its zoom.}
  \label{fig:graphzoom}
\end{figure}


\subsection*{Acknowledgements}
The authors would like to thank Ram Band for interesting discussions and helpful suggestions regarding the manuscript. G.B. acknowledges partial support from the NSF under grant
DMS-1410657.  G.C. acknowledges the support of NSERC grant RGPIN-2017-04259.  J.L.M. was supported in part by  NSF Applied Math Grant DMS-1312874 and NSF CAREER Grant DMS-1352353.

\section{The spectral flow}\label{sec:flow}

We now describe in more detail the spectral flow mentioned in the introduction, in the process giving a new proof of \eqref{def:simple} and \eqref{def:degenerate}, and setting the stage for our analysis of the rectangle.


Consider the Schr\"odinger operator $L = -\Delta + V$ on a bounded, Lipschitz domain $\Omega$, with Dirichlet boundary conditions. Let $\lambda_* \in \spec(L)$, and suppose $\phi_*$ is an eigenfunction for $\lambda_*$, with nodal set $\Gamma = \{x \in \Omega : \phi_*(x) = 0\}$. Throughout this section we impose the following assumption.

\begin{hyp}
Each nodal domain of $\phi_*$ has Lipschitz boundary.
\end{hyp}

For $n=2$ the hypothesis is always satisfied (see \cite[Theorem 2.5]{C76}) but its validity appears to be unknown in higher dimensions. In the absence of this assumption one can still define the Dirichlet-to-Neumann maps, following \cite{AE11}, but it is not immediately clear that they will have compact resolvent, and so the spectral flow argument becomes more complicated. We do not pursue this technical issue in the current paper.

We define a family of selfadjoint operators $L_\sigma$ via the bilinear forms
\begin{align}
	B_\sigma(u,v) = \int_\Omega \left[\nabla u \cdot \nabla v + Vuv\right] + \sigma \int_\Gamma uv
\end{align}
on $H^1_0(\Omega)$ for any $\sigma \in [0,\infty)$, and let $L_\infty$ denote the operator with Dirichlet boundary conditions on $\pO \cup \Gamma$. We denote by $\{\gamma_k(\sigma)\}$ the analytic eigenvalue branches for $L_\sigma$.  We first describe the relationship between these eigenvalue curves and the spectrum of $\Lambda_+ (\epsilon) + \Lambda_- (\epsilon)$.

\begin{lemma}\label{lem:DtN1}
For $\epsilon$ sufficiently small, the value $-\sigma$ is an eigenvalue of $\Lambda_+ (\epsilon) + \Lambda_- (\epsilon)$ if and only if $\lambda_* + \epsilon = \gamma_k(\sigma)$ for some $k \in \bbN$.
\end{lemma}

\begin{proof}
First suppose $-\sigma$ is an eigenvalue of $\Lambda_+ (\epsilon) + \Lambda_- (\epsilon)$, with eigenfunction $f \in H^{1/2}(\Gamma)$. Then there is a function $u \in H^1_0(\Omega)$ such that $u\big|_\Gamma = f$, with
\begin{align}\label{eqn:evalue}
	-\Delta u + Vu = (\lambda_* + \epsilon)u
\end{align}
in $\Omega \setminus \Gamma$ and
\begin{align}\label{eqn:Steklov}
	\frac{\p u}{\ \p \nu_+} + \frac{\p u}{\ \p \nu_-} + \sigma u  = 0
\end{align}
on $\Gamma$. This means $\lambda_* + \epsilon$ is an eigenvalue of $L_\sigma$, and so $\lambda_* + \epsilon = \gamma_k(\sigma)$ for some $k \in \bbN$.

Conversely, suppose $\lambda_* + \epsilon = \gamma_k(\sigma)$ for some $k$. The corresponding eigenfunction $u$ will by definition satisfy \eqref{eqn:evalue} in $\Omega \setminus \Gamma$ and \eqref{eqn:Steklov} on $\Gamma$, and so $f := u\big|_\Gamma$ satisfies the eigenvalue equation
\[
	\Lambda_+(\epsilon)f + \Lambda_-(\epsilon)f + \sigma f = 0.
\]
To complete the proof we must show that $f$ is not identically zero on $\Gamma$. If this was the case, $\lambda_* + \epsilon$ would be an eigenvalue of $L_\infty$, which is not possible because $\lambda_*$ is the first eigenvalue of $L_\infty$ and $\epsilon >0$ can be taken sufficiently small such that $\lambda_* + \epsilon$ lies in the spectral gap.
\end{proof}

Motivated by this result, we make the following definition.

\begin{define}
An eigenvalue curve $\gamma_k(\sigma)$ is said to give rise to a negative eigenvalue of $\Lambda_+(\epsilon) + \Lambda_-(\epsilon)$ if $\gamma_k(\sigma) = \lambda_* + \epsilon$ for some $\sigma > 0$.
\end{define}


Lemma \ref{lem:DtN1} says that $-\sigma$ is a negative eigenvalue of $\Lambda_+ (\epsilon) + \Lambda_- (\epsilon)$ if and only if there is an eigenvalue curve $\gamma_k(\sigma)$ that gives rise to it. In other words, the Morse index of $\Lambda_+ (\epsilon) + \Lambda_- (\epsilon)$, and hence the nodal deficiency of the eigenfunction $\phi_*$, is completely determined by the curves $\{\gamma_k(\sigma)\}$. Determining whether or not a given curve intersects $\lambda_* + \epsilon$ for some $\sigma>0$ is simplified by the following monotonicity result, which says that one simply needs to check the endpoints $\gamma_k(0)$ and $\gamma_k(\infty)$.



\begin{lemma}\label{gamma:mon}
If $u_k(\sigma)$ is analytic curve of normalized eigenfunctions for $\gamma_k(\sigma)$, then
\begin{align}\label{lambdaprime}
	\gamma_k'(\sigma) = \int_\Gamma u_k(\sigma)^2.
\end{align}
If $\gamma_k(0) \in \spec(L_\infty)$, then $\gamma_k(\sigma)$ is constant; otherwise $\gamma_k(\sigma)$ is strictly increasing.
\end{lemma}

The existence of an analytic curve of eigenfunctions for $\gamma_k(\sigma)$ is a consequence of the selfadjointness of $L_\sigma$; see \cite{K76}.

\begin{proof}
To simplify the notation we fix a value of $k$ and let $\gamma = \gamma_k(\sigma)$ and $u = u_k(\sigma)$. The eigenvalue equation $L_\sigma u = \gamma u$ is satisfied if and only if
\begin{align}\label{eqn:eigenvalue}
	B_\sigma(u,v) = \gamma \left<u,v\right>
\end{align}
for all $v \in H^1_0(\Omega)$, where $\left<\cdot,\cdot\right>$ denotes the $L^2(\Omega)$ inner product. Differentiating \eqref{eqn:eigenvalue} with respect to $\sigma$, we find that
\begin{align}\label{Bprime}
	B_\sigma'(u,v) + B_\sigma(u',v) =  \gamma' \left<u,v\right> +  \gamma \left<u',v\right>.
\end{align}
On the other hand, letting $v = u'$ in \eqref{eqn:eigenvalue} leads to $B_\sigma(u,u') =  \gamma \left<u,u'\right>$, and so, evaluating \eqref{Bprime} at $v=u$, we obtain 
\[
	 \gamma'  = B_\sigma'(u,u) = \int_\Gamma u^2
\]
as desired.

If $\gamma_k(0) \in \spec(L_\infty)$, the associated eigenfunction $u_k(0)$ vanishes on $\Gamma$, so \eqref{eqn:Steklov} is satisfied for any value of $\sigma$, and hence $\gamma_k(0) \in \spec(L_\sigma)$. The analyticity of the eigenvalue curves then implies $\gamma_k(0) = \gamma_k(\sigma)$ for all $\sigma$. 

If $\gamma_k(\sigma)$ is not strictly increasing, then $\gamma_k'(\sigma_0) = 0$ for some $\sigma_0$. From \eqref{lambdaprime} we infer that the associated eigenfunction vanishes on $\Gamma$, hence $\gamma_k(\sigma_0) \in \sigma(L_\infty)$. The argument in the previous paragraph now implies that $\gamma_k(\sigma)$ is constant, with $\gamma_k(0) = \gamma_k(\sigma_0) \in \spec(L_\infty)$.
\end{proof}

Using Lemmas \ref{lem:DtN1} and \ref{gamma:mon}, we can now verify \eqref{def:simple} and \eqref{def:degenerate}.  Indeed, let $\{\lambda_n(\sigma)\}$ denote the ordered eigenvalues of
$L_\sigma$, which are nondecreasing. As $\sigma \to \infty$ they
converge to the ordered eigenvalues of $L_\infty$, which by definition
has Dirichlet boundary conditions on $\pO \cup \Gamma$ (cf. \cite{AM12}). Since
$\lambda_*$ is the first Dirichlet eigenvalue on each nodal domain of
$\phi_{*}$, and hence is simple (on each domain), we have that the first
eigenvalue of $L_\infty$ is $\lambda_{*}$, with multiplicity
$\nod(\phi_*)$. It follows that
\begin{equation}
  \label{eq:spec_L_infty1}
	\lim_{\sigma \to \infty} \lambda_n(\sigma) = \lambda_{*}, \quad 1 \leq n \leq \nod(\phi_{*})
\end{equation}
and
\begin{equation}
  \label{eq:spec_L_infty2}
	\lim_{\sigma \to \infty} \lambda_n(\sigma) > \lambda_{*}, \quad n > \nod(\phi_{*}).
\end{equation}

If $\lambda_*$ is simple, $L_0$ has precisely $k_*$ eigenvalues $\lambda \leq \lambda_*$. Since the first $\nod(\phi_*)$ of these converge to $\lambda_*$ as $\sigma \to \infty$, the remaining $k_* - \nod(\phi_*)$ will converge to values greater than $\lambda_*$. Choosing $\epsilon>0$ sufficiently small, we conclude that each of these $k_* - \nod(\phi_*)$ eigenvalue curves passes through $\lambda_* + \epsilon$ for some finite $\sigma>0$, in the process giving rise to a negative eigenvalue of $\Lambda_+(\epsilon) + \Lambda_-(\epsilon)$. This verifies \eqref{def:simple}. Similarly, if $\lambda_*$ is degenerate, and we define $k_* = \min \{n \in \bbN : \lambda_n = \lambda_*\}$, then $L_0$ will have precisely $k_* - 1 + \dim \ker (\Delta + \lambda_*)$ eigenvalues $\lambda \leq \lambda_*$, and so $k_* - 1 + \dim \ker (\Delta + \lambda_*) - \nod(\phi_*)$ of them will pass through $\lambda_* + \epsilon$ as $\sigma$ increases from 0 to $\infty$. This verifies \eqref{def:degenerate}.

\section{The one-dimensional case}\label{sec:1d}

We now refine the general results of Section \ref{sec:flow} in the one-dimensional case. Let $\{Z_i\}_{i=1}^m$ be a partition of the interval $[0,\ell]$, so that
\[
	0 < Z_1 < \cdots < Z_m < \ell.
\]
For this partition, and some constant $\sigma \in \bbR$, we define a selfadjoint operator $L_\sigma$ by
\[
	L_\sigma = -\frac{d^2}{dx^2} + q(x)
\]
together with the boundary conditions
\begin{align}
	u(0) &= u(\ell) = 0  \label{BC:dir}\\
	u( Z_i-) &= u( Z_i+)    \label{BC:cont}\\
	u' ( Z_i+) - u' ( Z_i-) &= \sigma u ( Z_i) \ \ \text{for all} \ 1 \leq i \leq m. \label{BC:jump}
\end{align}


Let $\{\lambda_n(\sigma)\}$ denote the ordered eigenvalues of
$L_\sigma$.  As $\sigma\to\infty$, these converge to the ordered eigenvalues of
$L_\infty$, which has Dirichlet conditions imposed at each $Z_i$.  
Moreover, if $\{Z_i\}$ is the nodal partition of some Dirichlet eigenfunction $\phi_* = \phi_{k_*}$,
the spectrum of $L_\infty$ consists of $\lambda_* = \lambda_{k_*}(0)$, with
multiplicity $\nod(\phi_*)$, and other eigenvalues strictly greater
than $\lambda_*$.

We also know
from Lemma \ref{gamma:mon} that $\lambda_{k_*}(\sigma)$ is constant.  Since each $\lambda_n(\sigma)$ is simple and nondecreasing, this implies
\begin{align}\label{1Dlimita}
	\lim_{\sigma \to \infty} \lambda_n (\sigma) = \lambda_*, \quad 1 \leq n \leq k_*
\end{align}
and
\begin{align}\label{1Dlimitb}
	\lim_{\sigma \to \infty} \lambda_n (\sigma) > \lambda_*, \quad n > k_*.
\end{align}
This behavior is illustrated in Figure \ref{fig:lambdan}. 

Comparing \eqref{1Dlimita} and \eqref{1Dlimitb} to \eqref{eq:spec_L_infty1} and \eqref{eq:spec_L_infty2}, it follows that $\nod(\phi_*) = k_*$, and so we obtain Sturm's theorem as a consequence of the monotonicity and simplicity of the eigenvalues of $L_\sigma$ in the one-dimensional case.

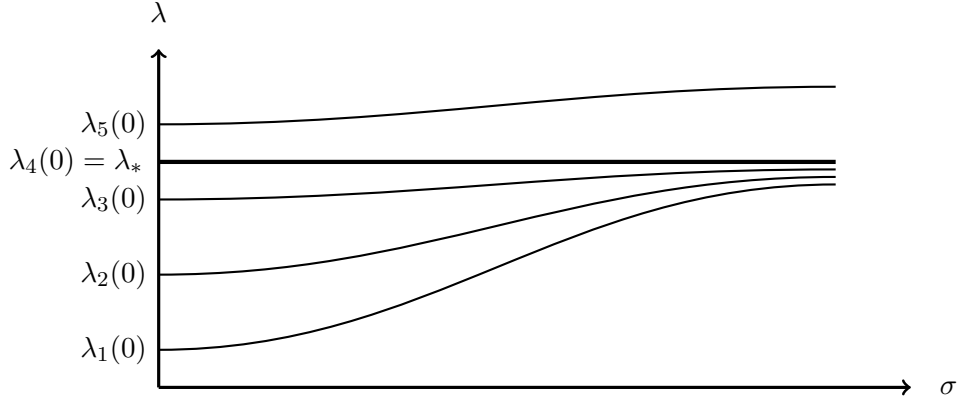
\begin{figure}
\begin{tikzpicture}[scale=1]
	\draw[very thick,->] (0,0) -- (10,0); 
	\node at (10.5,0) {$\sigma$};
	\draw[very thick,->] (0,0) -- (0,4.5); 
	\node at (0,5) {$\lambda$};	
	\draw[ultra thick] (0,3) -- (9,3); 
		\node at (-1.1,3) {$\lambda_4(0) = \lambda_*$};
	\draw[thick] (0,0.5) to [out=0, in=180] (9,2.7); 
		\node at (-0.6,0.5) {$\lambda_1(0)$};
	\draw[thick] (0,1.5) to [out=0, in=180] (9,2.8); 
		\node at (-0.6,1.5) {$\lambda_2(0)$};
	\draw[thick] (0,2.5) to [out=0, in=180] (9,2.9); 
		\node at (-0.6,2.5) {$\lambda_3(0)$};		
	\draw[thick] (0,3.5) to [out=0, in=180] (9,4); 
	\node at (-0.6,3.5) {$\lambda_5(0)$};
\end{tikzpicture}
\caption{The behavior of the first four eigenvalues of $L_\sigma$ in one dimension, with $k_* = 4$. The fourth eigenvalue, $\lambda_4(\sigma) = \lambda_*$, is constant, whereas the first three strictly increase to $\lambda_*$, and the fifth converges to some number strictly greater than $\lambda_*$, as claimed in \eqref{1Dlimita} and \eqref{1Dlimitb}.}
\label{fig:lambdan}
\end{figure}

\section{The rectangle}\label{sec:rectangle}
\label{rectangles}
We now return to the rectangle $[0,\alpha\pi] \times [0,\pi]$, considering a Schr\"odinger operator
\[
	L = -\Delta + q(x)+ r(y)
\]
with separable potential, where $q \in L^\infty(0,\alpha\pi)$ and $r \in L^\infty(0,\pi)$. Let $\{\lambda_m^x\}$ and $\{\lambda_n^y\}$ denote the Dirichlet eigenvalues for $-(d/dx)^2 + q(x)$ and $-(d/dy)^2 + r(y)$, respectively. The Dirichlet spectrum of $L$ is then given by
\[
	\spec(L) = \big\{\lambda_m^x + \lambda_n^y : m,n \in \bbN \big\}.
\]
For convenience we let $\lambda_{mn} = \lambda_m^x + \lambda_n^y$. Now suppose $\lambda_* = \lambda_{m_* n_*} \in \spec(L)$, and let $\Gamma$ denote the nodal set of the corresponding eigenfunction. As above, we define the family $\{L_\sigma\}$ of selfadjoint operators, with analytic eigenvalue curves $\{\gamma_k(\sigma)\}$. Note that $\{\gamma_k(0)\}$ are the eigenvalues of $L$, so for any $(m,n)$ there exists a $k = k(m,n)$ with $\gamma_k(0) = \lambda_{mn}$.

\begin{define}
A lattice point $(m,n)$ is said to give rise to a negative eigenvalue of $\Lambda_+(\epsilon) + \Lambda_-(\epsilon)$ if the curve $\gamma_k(\sigma)$ does, where $k = k(m,n)$ as above.
\end{define}

Our main result, generalizing the picture in Figure \ref{fig:ellipse}, is the following.

\begin{theorem}\label{thm:main}
The point $(m,n)$ gives rise to a negative eigenvalue of $\Lambda_+(\epsilon) + \Lambda_-(\epsilon)$ if and only if  $\lambda_{mn} \leq \lambda_*$ and either $m > m_*$ or $n > n_*$.
\end{theorem}

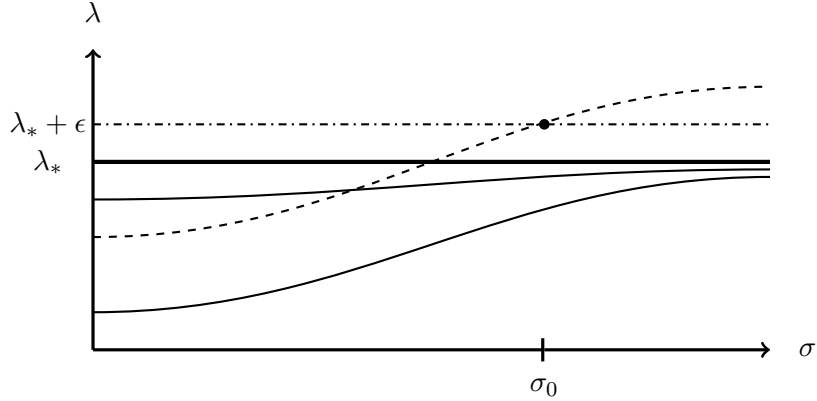
\begin{figure}
\begin{tikzpicture}[scale=1]
	\draw[very thick,-|] (0,0) -- (6,0); 
	\draw[very thick,->] (6,0) -- (9,0); 
	\node at (9.5,0) {$\sigma$};
	\draw[very thick,->] (0,0) -- (0,4); 
	\node at (0,4.5) {$\lambda$};	
	\draw[ultra thick] (0,2.5) -- (9,2.5); 
		\node at (-0.6,2.5) {$\lambda_*$};
	\draw[thick,dashdotted] (0,3) -- (9,3); 
		\node at (-0.6,3) {$\lambda_* + \epsilon$};
	\draw[thick] (0,0.5) to [out=0, in=180] (9,2.3); 
	\draw[thick,dashed] (0,1.5) to [out=0, in=180] (9,3.5); 
	\draw[thick] (0,2) to [out=0, in=180] (9,2.4); 
	\node at (6,-0.5) {$\sigma_0$};
	\fill (6,3) circle[radius=2pt]; 
\end{tikzpicture}
\caption{The behavior of $\gamma_k(\sigma)$ as $\sigma\to\infty$.  The dashed curve has $\gamma_{k}(0) = \lambda_{mn} < \lambda_*$ and $\gamma_k(\infty) > \lambda_*$, and hence generates a negative eigenvalue $-\sigma_0$ for $\Lambda_+ (\epsilon) + \Lambda_- (\epsilon)$. The other two eigenvalues curves correspond to $(m,n) \leq (m_*, n_*)$, and hence stay below $\lambda_*$ for all finite values of $\sigma$.}
\label{fig:lambdamn}
\end{figure}

That is, the eigenvalue curve $\gamma_k(\sigma)$, with initial value
$\gamma_k(0) = \lambda_{mn}$, crosses $\lambda_* + \epsilon$ for some
finite, positive value of $\sigma$ if and only if $m$ and $n$ satisfy
the given conditions.  In the case $V \equiv 0$ these conditions
reduce to $(m,n) \in \overline{E_*} \setminus R_*$, as promised in the
Introduction.

\begin{proof}
Let $k = k(m,n)$, so that $\gamma_k(0) = \lambda_{mn} = \lambda_m^x + \lambda_n^y$.

Given $\lambda_* = \lambda^x_{m_*} + \lambda^y_{n_*}$ as above, let $u^x_{m_*}$  and $u^y_{n_*}$ denote the eigenfunctions for $\lambda_{m_*}^x$ and $\lambda_{n_*}^y$, with nodal sets $\{Z^x_k\} \subset (0,\alpha\pi)$ and $\{Z^y_k\} \subset (0,\pi)$, respectively. With respect to these nodal partitions, we define operators $L^x_\sigma$ and $L^y_\sigma$, as in Section \ref{sec:1d}, for $\sigma \in \bbR$. Denoting the eigenvalues by $\{\lambda_m^x(\sigma)\}$ and $\{\lambda_n^y(\sigma)\}$, we have $\gamma_k(0) = \lambda_m^x(0) + \lambda_n^y(0)$, hence
\begin{align}\label{lambdamn}
	\gamma_k(\sigma) = \lambda^x_m(\sigma) + \lambda^y_n(\sigma)
\end{align}
for all $\sigma$.

Since $\gamma_k(\sigma)$ is nondecreasing, the equation $\gamma_k(\sigma) = \lambda_* + \epsilon$ will be satisfied for some $\sigma>0$ if and only if $\gamma_k(0) \leq \lambda_*$ and $\gamma_k(\sigma) > \lambda_*$ for sufficiently large $\sigma$. The condition $\gamma_k(0) \leq \lambda_*$ is equivalent to $\lambda_{mn} \leq \lambda_*$. On the other hand, it follows from \eqref{1Dlimita} and \eqref{1Dlimitb} that
\[
	\lim_{\sigma\to\infty} \lambda_m^x(\sigma) > \lambda^x_{m_*}
\]
if and only if $m > m_*$, and similarly for the limit of $\lambda_n^y(\sigma)$, hence
\[
	\lim_{\sigma\to\infty} \gamma_k(\sigma) > \lambda^x_{m_*} + \lambda^x_{n_*} = \lambda_*
\]
holds if and only if either $m > m_*$ or $n > n_*$ (see
Figure~\ref{fig:lambdamn} for an example).
\end{proof}

\appendix

\section{Example for a Rectangle}
\label{app:rec}

Let us consider first the one-dimensional eigenvalue problem for the case $q(x) = 0$ from Section \ref{sec:1d}.  Namely, we wish to compute the eigenvalues $\{ \lambda_n (\sigma) \}$ for $\sigma \geq 0$.  

\subsection{$\{ Z_k \} = \{  \frac12 \ell \}$}

The second Dirichlet eigenfunction for the Laplacian the interval $[0,\ell]$ has a zero at $\ell/2$.  Using this nodal point to define the boundary conditions in $\sigma$, as in Section \ref{sec:1d}, we look for the eigenvalues $\lambda_n (\sigma)$.  We will use the notation $\lambda_n (\sigma; 2)$ to denote the $n$th eigenvalue that arises from the spectral flow in $\sigma$ set at the nodal point of the second Dirichlet eigenfunction.  Symmetry considerations guarantee that the corresponding lowest eigenfunction, $u_1 (x)= u_1 (x, \sigma; 2)$, is symmetric with respect to $\ell/2$.  The eigenvalues $\lambda_n (\sigma; 2)$ in this case can be found by taking $u_1 (x) = \sin (\kappa x)$ on $[0, \ell/2]$ for $\kappa^2 = \lambda_n$.  Condition \eqref{BC:jump} gives
\[
-2 u_1' \left( \frac{\ell}{2} \right) = \sigma u_1 \left( \frac{\ell}{2} \right),
\]
and hence
\begin{equation}
\label{eq:Zk1_2}
\sigma = - 2 \kappa \cot \left( \kappa \frac{\ell}{2} \right).
\end{equation}
Thus, $\lambda_1 (\sigma; 2) = \kappa^2$ is given as the implicit solution to \eqref{eq:Zk1_2} for finding the lowest eigenvalue.  

\subsection{$\{ Z_k \} = \{ \frac13 \ell, \frac23 \ell \}$}

Now, let us consider the next excited state, or the case of the nodal
set given by $2$ zeros equidistributed throughout the interval.  As before the lowest eigenfunction of $L_\sigma$, denoted $u_1 (x)= u_1 (x, \sigma; 3)$, is symmetric with respect to $\ell/2$ and we can write
\begin{equation*}
u_1 (x)= \left\{  \begin{array}{cc}
a \sin (\kappa x), & x \in [0, \ell/3] \\
b \cos (\kappa (\ell/2-x)) &  x \in [\ell/3, \ell/2]
\end{array} \right.
\end{equation*}
Hence, conditions \eqref{BC:cont} and \eqref{BC:jump} at $\ell/3$ imply
\begin{align*}
 a \sin \left(  \frac{ \kappa \ell}{3} \right) = b \cos \left(  \frac{ \kappa \ell}{6} \right)  & = c , \\
- \left(  a \kappa \cos \left(  \frac{ \kappa \ell}{3} \right)  - b \kappa \sin \left(  \frac{ \kappa \ell}{6} \right) \right) & = \sigma c,
\end{align*}
for $c = u_1 (\ell/3)$.  Solving out for $c$, we arrive at
\[
\sigma = \kappa \left(  \tan \left(  \frac{ \kappa \ell}{6} \right)  - \cot \left(  \frac{ \kappa \ell}{3} \right)  \right),
\]
which can be solved implicitly for $\lambda_1 (\sigma; 3) = \kappa^2$.  

A similar approach applies to find the second eigenfunction $u_2 (x) = u_2 (x,\sigma; 3)$, which is anti-symmetric with respect to $\ell/2$.  Following the same logic, we arrive at
\[
\sigma = -\kappa \left(  \cot \left(  \frac{ \kappa \ell}{3} \right)  + \cot \left(  \frac{ \kappa \ell}{6} \right)  \right),
\]
which can be solved implicitly for $\lambda_2 (\sigma; 3) = \kappa^2$.  

\subsection{An example with nodal deficiency $3$ on the rectangle}

Let us now consider a rectangle of the form $[0,\pi] \times [0, \alpha
\pi]$ with $\alpha < 1$ but such that $1-\alpha \ll 1$.  We observe in
this case that for the Laplacian with Dirichlet boundary conditions, 
\[
1^2 + \left( \frac{1}{\alpha} \right)^2 = \lambda_{1,1}  < \lambda_{2,1} < \lambda_{1,2} < \lambda_{2,2} < \lambda_{3,1} < \lambda_{1,3} = 1^2 + \left( \frac{3}{\alpha} \right)^2.
\]
Therefore the sixth eigenvalue $\lambda_6 = \lambda_{1,3}$ has $3$
nodal domains and therefore nodal deficiency $3$, see Figure \ref{fig:a1}.

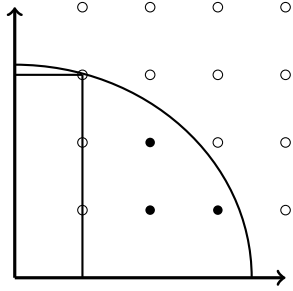
\begin{figure}[!tbp]
	\begin{tikzpicture}[scale=0.9]
		\draw[thick] (0,3) -- (1,3); 
		\draw[thick] (1,0) -- (1,3); 
		\draw[thick] (3.5,0) arc (0:90:3.5 and 3.15);
		\draw[very thick,->] (0,0) -- (4,0); 
		\draw[very thick,->] (0,0) -- (0,4); 
		\draw (1,1) circle[radius=2pt];
		\draw (1,2) circle[radius=2pt];
		\draw (1,3) circle[radius=2pt];
		\draw (1,4) circle[radius=2pt];
		\fill (2,1) circle[radius=2pt];
		\fill (2,2) circle[radius=2pt];
		\draw (2,3) circle[radius=2pt];
		\draw (2,4) circle[radius=2pt];
		\fill (3,1) circle[radius=2pt];
		\draw (3,2) circle[radius=2pt];
		\draw (3,3) circle[radius=2pt];
		\draw (3,4) circle[radius=2pt];
		\draw (4,1) circle[radius=2pt];
		\draw (4,2) circle[radius=2pt];
		\draw (4,3) circle[radius=2pt];
		\draw (4,4) circle[radius=2pt];
	\end{tikzpicture}
\caption{Illustrating the nodal deficiency count for the rectangle example in the Appendix.}
\label{fig:a1}
\end{figure}

Setting $\lambda_* = \lambda_6 = \lambda_{1,3}$, we obtain the spectral flow 
\begin{align*}
	\gamma_6 (\sigma) &= \lambda_*, \\
	\gamma_5 (\sigma)& = 3^2 + \lambda_1^y (\sigma; 3), \\
	\gamma_4 (\sigma) &= 2^2 + \lambda_2^y (\sigma; 3), \\
	\gamma_3 (\sigma) &= 1^2 + \lambda_2^y (\sigma; 3), \\
	\gamma_2 (\sigma) &= 2^2 + \lambda_1^y (\sigma; 3), \\
	\gamma_1 (\sigma) &= 1^2 + \lambda_1^y (\sigma; 3),
\end{align*}
which was the flow depicted on Figure~\ref{fig:qgraphrect}(right).
The above equations can be analyzed to show that
$\gamma_2, \gamma_4, \gamma_5$ all cross $\gamma_6$ as
$\sigma \to \infty$, whereas $\gamma_1$ and $\gamma_3$ do not.

\bibliographystyle{amsplain}
\bibliography{maslov}
 
\end{document}